\documentclass[11pt]{article}%
\pdfoutput=1
\usepackage{graphicx}
\usepackage{amssymb}
\usepackage{amsmath}
\usepackage{amsfonts}
\usepackage{amsthm,enumerate}%
\providecommand{\U}[1]{\protect\rule{.1in}{.1in}}
\pdfoutput=1
\newtheorem {theorem}{Theorem}[section]
\newtheorem {proposition}[theorem]{Proposition}
\newtheorem {lemma}[theorem]{Lemma}

\begin{document}

\title{Ricci flow and curvature on the variety of flags on the two dimensional
projective space over the complexes, quaternions and the octonions.}
\author{Man-Wai Cheung
\and Nolan R. Wallach}
\maketitle

\begin{abstract}
For homogeneous metrics on the spaces of the title it is shown that the Ricci
flow can move a metric of stricly positive sectional curvature to one with
some negative sectional curvature and one of positive definite Ricci tensor to
one with indefinite signature.

\end{abstract}

\section{Introduction}

In this note, we will show that a metric of the homogeneous Riemannian
manifold $SU(3)/T^{2}$ with strictly positive curvature is deformed to a
metric with some negative sectional curvature by the Ricci flow. This result
has been announced by B\"ohm and Wilking [BW] in which they assert that this
can be proved using a method similar to the one they used to show that a
metric of positive sectional curvature on $Sp(3)/Sp(1)\times Sp(1)\times
Sp(1)$ can be flowed to one with some negative Ricci curvature. We use a
different and simpler approach to prove this result. We will show that if we
initiate the Ricci flow at a metric on the boundary of the metrics with
positive sectional then the derivative of the flow of sectional curvature at a
plane of zero curvature is negative for the all of the examples in [W]. We
also show that for all the examples in [W] (dimension 6,12,24) the Ricci flow
can cause the Ricci tensor to go from positive definite to signature $(d,2d)$
($d=2,4,8)$. In the concluding remarks to the paper we give a simple variant
of Valiev's necessary and sufficient condition for a homogeneous metric on one
of the spaces to have strictly positive sectional curvature. We would like to
thank Lei Ni for suggesting that we look at the curvature transition of the
$6$ dimensional space in [W] under the Ricci flow and for his patience as we
were trying to understand the subtleties of the argument in [BW] for the
twelve dimensional example.

\section{Setup}

In this section, we will set up the notation for the main calculations and
establish the Ricci flow equations in terms of the metric parameters. Set
$G=SU(3),Sp(3)$ or compact $F_{4}$ and let $K$ be respectively a maximal
torus, $T^{2}$, of $SU(3)$, $Sp(1)\times Sp(1)\times Sp(1)$ in $Sp(3)$ or
$Spin(8)$ in compact $F_{4}$. Let ${\mathfrak{g}}$ be the Lie algebra of $G$,
${\mathfrak{k}}$ be the Lie algebra of a $K$. Let ${\mathfrak{p}}$ be the
$Ad(K)$--invariant complement to ${\mathfrak{k}}$ in ${\mathfrak{g}}$. Then
${\mathfrak{p}}$ can be decomposes into a direct sum of three irreducible
inequivalent $K$-invariant subspaces ${\mathfrak{p}}=V_{1}\oplus V_{2}\oplus
V_{3}$. Consider the Ad$(G)$-invariant inner product $\left\langle
X,Y\right\rangle _{0}=-1/2$ Re tr$(X,Y)$ on ${\mathfrak{g}}$ for the first two
examples and in the case of $G=F_{4}$ the unique $Ad(G)$-invariant inner
product that agrees with our choice for the imbedded $Sp(3)$ that is
compatible with the decompositions. The dimensions of the real vector spaces
$V_{i}$ are the same in each case and are respectively $d=2,4,$or $8.$ In each
case we may identify the spaces $V_{i}$ with the fields over $\mathbb{R} $:
$\mathbb{C}$, $\mathbb{H}$(the quaternions), $\mathbb{O}$(the octonions) such
that the inner product is $\operatorname{Re}(z\overline{w})$. If $z\in
V_{1},w\in V_{2}$ then $[z,w]\in V_{3}$ and under our identification
corresponds to $\overline{z}w$ in $V_{3}.$Similarly with sign changes as in
the cross-product $[V_{i},V_{j}]\in V_{k}$ if $i,j,k$ are distinct.. Schur's
Lemma implies that any $K$-invariant inner product on $\mathfrak{p}$ is given
by
\begin{equation}
x_{1}\left\langle ...,...\right\rangle _{0}|_{V_{1}}+x_{2}\left\langle
...,...\right\rangle _{0}|_{V_{2}}+x_{3}\left\langle ...,...\right\rangle
_{0}|_{V_{3}}%
\end{equation}
where $x_{1},x_{2},x_{3}$ are positive constants. Let $g$ be the Riemannian
structure on $M$ corresponding to $(x_{1},x_{2},x_{3})$. We will write
$g\longleftrightarrow(x_{1},x_{2},x_{3})$. In [AW] it was proved that if
$x_{1}=x_{2}=1$ then for all examples above the sectional curvature is
strictly positive if $0<x_{3}<1$ or $1<x_{3}<\frac{4}{3}$. We note

\begin{lemma}
If $x_{1}=x_{2}$ then the sectional curvature is is strictly positive if
$0<\frac{x_{3}}{x_{1}}<1$ or $1<\frac{x_{3}}{x_{1}}<\frac{4}{3}$ and there is
some strictly negative curvature if $\frac{x_{3}}{x_{1}}>\frac{4}{3}$.
\end{lemma}

\begin{proof}
We need only prove the assertion about negative curvature. We may assume that
$x_{1}=x_{2}=1.$We consider the embedding of $SU(3)$ into $G$ so that $T^{2}$
imbeds in $K$ and the the imbedding of the complement to $Lie(T^{2})$ in
$Lie(SU(3)$, $\mathfrak{q}$, imbeds in $\mathfrak{p}$ as $\mathbb{C}$ imbeds
in $\mathbb{H}$ or $\mathbb{O}$. We note that if $u,v\in$ $\mathfrak{q}$ then
the formula in Lemma 7.3 of [W] reduces the calculation to the case of
$SU(3).$We compute a specific curvature
\[
u=\left[
\begin{array}
[c]{ccc}%
0 & u_{1} & u_{2}\\
-u_{1} & 0 & u_{3}\\
-u_{2} & -u_{3} & 0
\end{array}
\right]  ,v=\left[
\begin{array}
[c]{ccc}%
0 & v_{1} & v_{2}\\
-v_{1} & 0 & v_{3}\\
-v_{2} & -v_{3} & 0
\end{array}
\right]
\]
with $u_{1}=1,v_{1}=-1,u_{2}=v_{2}=1/\sqrt{1+x^{2}},u_{3}=v_{3}=x/\sqrt
{1+x^{2}}$ with $x\in\mathbb{R}$. Then with $x_{1}=x_{2}=1,x_{3}=1+t$%
\[
g(R(u,v)v,u)=\frac{2}{1+x^{2}}(1-3t+(1+t)^{2}x^{2}).
\]
So if $t=\frac{1}{3}+s$ with $s>0$ and
\[
0<x<\sqrt{\frac{3s}{(1+(\frac{4}{3}+3s)^{2})}}%
\]
then the curvature corresponding to the two plane $span_{%
\mathbb{R}
}(u,v)$ is negative. This shows that the there is negative Gaussian curvature
for any $t>\frac{1}{3}$ so my condition is necessary and sufficient.
\end{proof}

We also note that Schur's lemma implies that the Ricci curvature of $g$,
denoted Ric$(g)$, is given by
\begin{equation}
\text{Ric}(g)=x_{1}r_{1}\left\langle ...,...\right\rangle _{0}|_{V_{1}}%
+x_{2}r_{2}\left\langle ...,...\right\rangle _{0}|_{V_{2}}+x_{3}%
r_{3}\left\langle ...,...\right\rangle _{0}|_{V_{3}}.
\end{equation}
Using the (first) Lemma 7.1 in [W] it is easily seen that $r_{i}$ is given by
\begin{equation}
r_{i}=\frac{dx_{i}^{2}-dx_{j}^{2}-dx_{k}^{2}+(10d-8)x_{j}x_{k}}{2x_{1}%
x_{2}x_{3}}%
\end{equation}
where $\{i,j,k\}=\{1,2,3\}.$

We note that the Ricci flow preserves left invariant metrics on the spaces
$G/K$ and hence can be considered to be the ordinary differential equation%
\begin{equation}
\frac{dx_{i}}{dt}=-2r_{i}x_{i}%
\end{equation}
In particular we see that the set of metrics with $x_{i}=x_{j}$ for some $i,j
$ is preserved by the Ricci flow. Also, permutation of the indices of the
$x_{i}$ preserves the solutions.

\section{The sectional curvature}

In this section we will prove that the Ricci flow deforms some metric $g$ with
strictly positive curvature into metric with some negative sectional
curvature. To start with, we investigate the metric $g_{0}\longleftrightarrow
(1,1,\frac{4}{3})$ which, in light of Lemma 2.1 is of nonnegative sectional
curvature and $g\longleftrightarrow(1,1,u)$, $u>\frac{4}{3}$ has some strictly
negative curvature. Using the symmetric invariance of the system (2.4) we note
that if we start with $g_{0}\longleftrightarrow(1,1,\frac{4}{3})$ under the
Ricci flow the metric $g_{t}\longleftrightarrow(x_{1}(t),x_{2}(t),x_{3}(t))$
satisfies $x_{1}(t)=x_{2}(t)$. Our strategy to prove that some curvature turns
negative is to show that
\begin{equation}
\frac{d}{dt}_{t=0}\frac{x_{3}(t)}{x_{1}(t)}>0.
\end{equation}
This will say that there exists $\varepsilon>0$ such that $\frac{1}%
{x_{3}(-\varepsilon)}g_{-\varepsilon}\longleftrightarrow(1,1,u)$ with
$1<u<\frac{4}{3}$ and $\frac{1}{x_{3}(\varepsilon)}g_{\varepsilon
}\longleftrightarrow(1,1,v)$ with $v>\frac{4}{3}$, So Lemma 2.1 implies our
assertion. We now carry out the calculation.%
\[
\frac{d}{dt}\frac{x_{3}(t)}{x_{1}(t)}=\frac{x_{3}^{\prime}(t)x_{1}%
(t)-x_{3}(t)x_{1}^{\prime}(t)}{x_{1}(t)^{2}}%
\]
so (2.4) implies that%
\begin{equation}
\frac{d}{dt}\frac{x_{3}(t)}{x_{1}(t)}=-2\frac{x_{3}(t)}{x_{1}(t)}\left(
r_{3}-r_{1}\right)  .
\end{equation}
In the three cases ($d=2,4,8$) we have $-2(r_{3}-r_{1})=-2+\frac{4d}{3}>0$.

We have proved

\begin{theorem}
On the three examples of [W] the Ricci flow deforms certain positively curved
metrics into metrics with mixed sectional curvatures.
\end{theorem}

We note that this result for the $12$ dimensional example follows from [BW].

\section{Change in Ricci curvature.}

We first indicate why the method of the last section doesn't work for Ricci
curvature. We consider the case when $x_{1}=x_{2}$ and calculate
\begin{equation}
2(r_{1}-r_{3})=\frac{-2(1-\frac{x_{3}}{x_{1}})((4d-4)-d\frac{x_{3}}{x_{1}}%
)}{x_{3}}.
\end{equation}
We therefor see in (light of (3.2)) that if $0<\frac{x_{3}(t)}{x_{1}(t)}<1$
then $\frac{d}{dt}\frac{x_{3}(t)}{x_{1}(t)}<0$. So if we started the Ricci
flow with a (positive curvature) initial condition $x_{1}=x_{2},\frac{x_{3}%
}{x_{1}}<1$ then $\frac{x_{3}}{x_{1}}$ is decreasing. If initially
$1<\frac{x_{3}}{x_{1}}<\frac{4(d-1)}{d}$ then under the flow we would have
$\frac{d}{dt}\frac{x_{3}(t)}{x_{1}(t)}>0$. If $\frac{4(d-1)}{d}<\frac{x_{3}%
}{x_{1}}$ then $\frac{d}{dt}\frac{x_{3}(t)}{x_{1}(t)}<0.$ Thus $\frac{x_{3}%
}{x_{1}}=1$ is a repelling (i.e unstable fixed point) and $\frac{x_{3}}{x_{1}%
}=$\ $\frac{4(d-1)}{d}$ is an attractor. The upshot is that if the initial
condition is $x_{1}=x_{2}$ and the sectional curvature is positive then
$\frac{x_{3}}{x_{1}}<\frac{4(d-1)}{d}$ for the entire Ricci flow. On the other
hand the Ricci tensor for $x_{1}=x_{2}$ is given by%
\[
\frac{(10d-8-d\frac{x_{3}}{x_{1}})}{2}(\left\langle \dots,\dots\right\rangle
_{0}|_{V_{1}}+\left\langle ...,...\right\rangle _{0}|_{V_{2}})+\frac
{(8d-8)-d\left(  \frac{x_{3}}{x_{1}}\right)  ^{2}}{2}\left\langle
...,...\right\rangle _{0}|_{V_{3}}.
\]
Thus if we begin the Ricci flow with a metric of positive curvature and
$x_{1}=x_{2}$ then $\frac{_{x_{3}}}{x_{1}}<\frac{4(d-1)}{d}$ which implies
that%
\[
\frac{(10d-8-d\frac{x_{3}}{x_{1}})}{2}>\frac{3d-2}{d}>0
\]
and
\begin{equation}
\frac{(8d-8)-d\left(  \frac{x_{3}}{x_{1}}\right)  ^{2}}{2}>\frac
{4(d-1)(3d-2)}{d}>0.
\end{equation}
This indicates how delicate the methods of [BW] must be.

We observe that if the initial condition satisfies $x_{2}>x_{1}>x_{3}$ then
the flow will stay among the homogeneous metrics satisfying this condition.
Since Ricci curvature is invariant under constant scalar multiples of the
metric we may assume that our initial metric corresponds to $x_{1}%
=1,x_{2}=1+r,x_{3}=s$ and $s<1$ (notice that Lemma 2.1 implies that if $s<1$
is fixed and $r$ is sufficiently small then the metric has positive sectional
curvature ). We also note that (2.3) implies that if $x_{1}=1,x_{2}%
=1+r,x_{3}=s$ then the coefficients of the Ricci curvature are given by%
\[
r_{1}x_{1}=\frac{-2rd-dr^{2}+(10d-8)s+(10d-8)rs-ds^{2}}{2(1+r)s},
\]%
\[
r_{2}x_{2}=\frac{dr+dr^{2}+(10d-8)s-ds^{2}}{2s}%
\]
and%
\[
r_{3}x_{3}=\frac{(8d-8)+(8d-8)r-dr^{2}+ds^{2}}{2(1+r)}.
\]
Thus if $s<1$ and $0<r<1$ then $r_{2}x_{2}$ and $r_{3}x_{3}$ are strictly
positive. If we solve the quadratic equation for $r_{1}x_{1}=0$ then we have
for the cases $d=2,4,8$ respectively%
\[
r=\sqrt{1+8s^{2}}-(1-3s),
\]%
\[
r=\sqrt{1+15s^{2}}-(1-4s)
\]
and%
\[
r=\sqrt{1+\frac{77}{4}s^{2}}-(1-\frac{9}{2}s).
\]
We note that if we substitute these values of $r$ into the above coefficients
of the Ricci tensor then \ we find that if $s<1$, $r_{2}x_{2}>0$ and
$r_{3}x_{3}>0$. So if we show that if we take our initial condition at such a
value $\frac{dr_{1}}{dt}<0$ we will have shown that the Ricci flow transitions
from positive definite to signature $(d,2d)$ ($d$ negatives). We therefore
study%
\[
-2\sum r_{i}x_{i}\frac{\partial r_{1}}{\partial x_{i}}%
\]
at these values we find that if $d=2$ then this expression is negative for
$0<s<1-\sqrt{\frac{5}{8}}$ $(0.20943058...)$ for $d=4$ the expression is
negative for $0<s<$ $\frac{30+5\sqrt{21}-3\sqrt{5(21+4\sqrt{21})}}{30}$
$(0.361437...)$ and for $d=8$ the expression is negative for%
\[
0<s<\frac{693+11\sqrt{2737}-7\sqrt{22(511+9\sqrt{2737})}}{616}%
\]
($0.389089...).$This proves

\begin{theorem}
For all the examples in [W] (i.e. the manifold of flags in the two
\ dimensional projective space over $\mathbb{C}$, $\mathbb{H}$ or $\mathbb{O}%
$)the Ricci flow of a metric with positive definite Ricci tensor can flow to
one with signature $(d,2d)$.
\end{theorem}

\section{Concluding remarks}

In this section we will give a necessary and sufficient condition that the
metric corresponding to $(x_{1},x_{2},x_{3})$ have positive curvature for the
three types of examples that we have been studying. We compare this condition
to what is necessary for positive Ricci curvature and one, thereby, gets a
better understanding of the result in [BW].

We first observe that the permutation action of the symmetric group permutes
the $(x_{1},x_{2},x_{3})$ that correspond to strictly positive curvature among
themselves. we have also completely described the $(x_{1},x_{2},x_{3})$ with
some pair $x_{i}=x_{j}$ with $i\neq j$. Thus we are left with the cases where
\[%
{\displaystyle\prod\limits_{i<j}}
(x_{i}-x_{j})\neq0.
\]
Using the action of the symmetric group just described we may assume that
\thinspace$x_{2}>x_{1}>x_{3}>0$ (we chose this order to be consistent with the
results of [BW]. Since a multiplication by a positive scalar doesn't change
the sign of curvature we may assume that $x_{1}=1,x_{2}=1+r$ and $x_{3}=s$
with $r>0$ and $s<1.$The following result follows directly from Theorem 3 a)
in [V].

\begin{proposition}
With the notation above a necessary and sufficient condition that the
sectional curvature be positive is $r<\frac{s-2+2\sqrt{1-s+s^{2}}}{3}$.
\end{proposition}

\noindent\textbf{Remark.} We note that if $0<s<1$ then
\[
\frac{s^{2}}{4}<\frac{s-2+2\sqrt{1-s+s^{2}}}{3}<\frac{s^{2}}{3}%
\]
and the expression estimated is monotone increasing. This can be seen in the
following graph:%
\begin{figure}[h]%
\centering
\includegraphics[
natheight=3.720200in,
natwidth=4.479800in,
height=2.1212in,
width=2.5504in
]%
{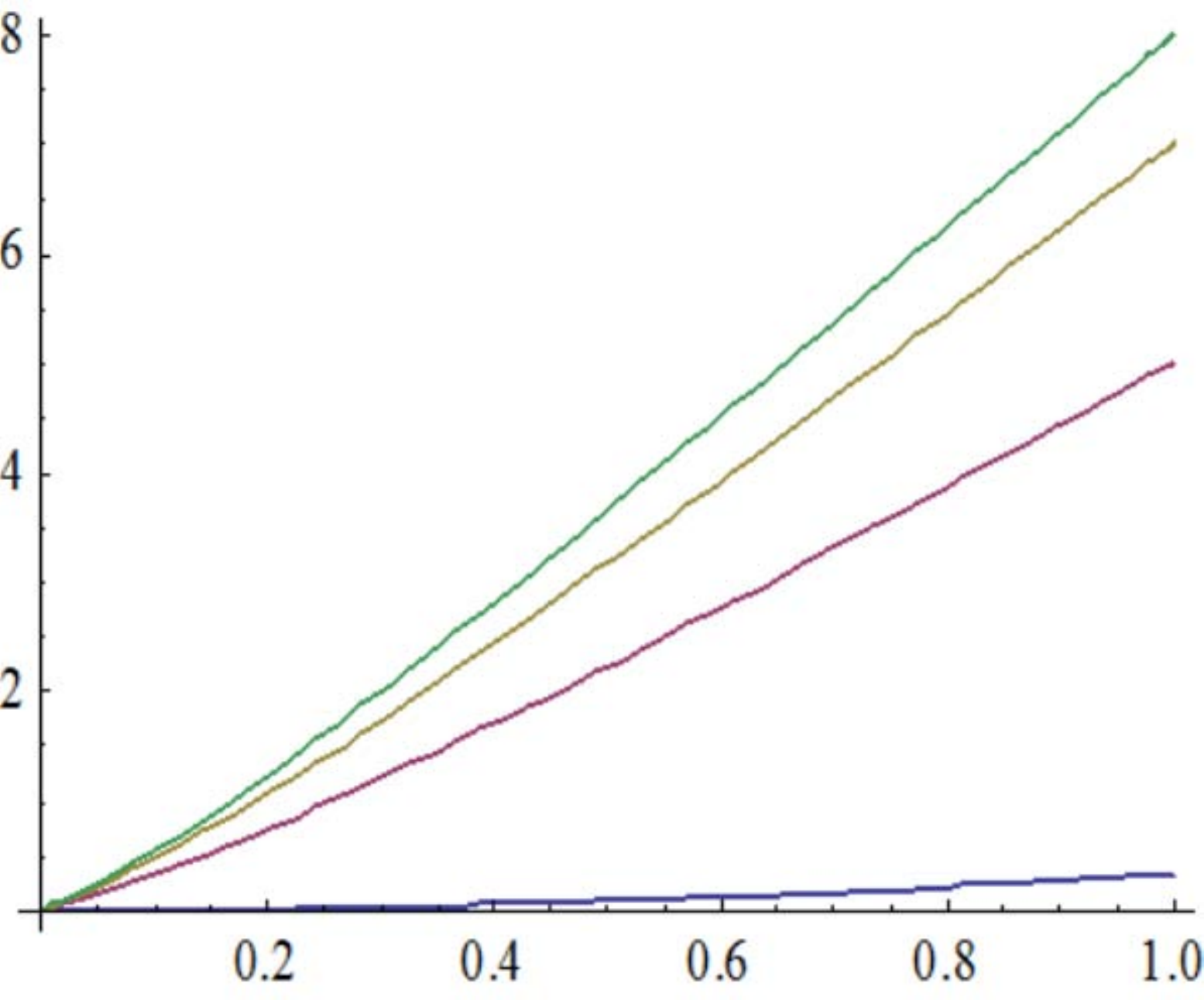}
\end{figure}

The axes are horizontal, $s$, and vertical, $r$, the set points under each
curve represent the $r$ values for each $s$ value such that $(1,1+r,s)$ with
$r>0$ and $0<s<1$ respectively satisfies the necessary condition above (lowest
curve, blue), the necessary and sufficient condition for positive Ricci
curvature for the $6$ (second curve, red), $12\ $(third curve,yellow) and $24$
(top curve,green) dimensional examples. We note that to get the full set of
metrics with strictly positive curvature satisfying the inequalities
$x_{3}\leq x_{1}\leq x_{2}$ one must allow the points $(s,0)$, $0<s<1$ and
$(1,r)$ with $0<r<\frac{1}{3}$ (that is add the the original set given in [AW]).

\bigskip

In the argument in [BW] they start their Ricci flow at a metric corresponding
to $(x_{1},x_{2},x_{3})$ such that $(x_{1},x_{2},x_{3})/x_{1}=(1,1+r,s)$ (in
our notation) and $r>0$, $0<s<1$ (the reason for our strange condition). In
light of the above $(r,s)$ must be below the blue curve in the graph above. In
the Ricci flow (normalized or not) the set $x_{2}>x_{1}>x_{3}$ is preserved.
If $(x_{1}(t),x_{2}(t),x_{3}(t))$ is a point in the flow\ and $(x_{1}%
(t),x_{2}(t),x_{3}(t))/x_{1}(t)=(1,1+r(t),s(t))$ then their the curve
$(s(t),r(t))$ starts at $t=0$ under the lowest (blue) curve (so as to have
positive curvature) and it must eventually cross the yellow (second highest)
curve in order to have some negative Ricci curvature.

\begin{center}
{\LARGE References}

\bigskip
\end{center}

\noindent\lbrack AW] Simon Aloff, Nolan R. Wallach,An infinite family of
distinct 7-manifolds admitting positively curved Riemannian structures, Bull.
Amer. Math. Soc.,81(1975),93-97.\smallskip

\noindent\lbrack BW]C. B\"ohm, B. Wilking, \textit{Nonnegatively curved
manifolds with finite fundamental groups admit metrics with positive Ricci
curvature}, Geom. funct. anal. 17(2007),665-681.\smallskip

\noindent\lbrack V]F.M. Valiev, \textit{Precise estimates for the sectional
curvatures of homogeneous Riemannian metrics on Wallach spaces}, Siberian
Math. Journal 20 (1979), 176-187.\smallskip

\noindent\lbrack W]Nolan R. Wallach, \textit{Compact homogeneous Riemannian
manifolds with strictly positive curvature}, Anna. of Math. (2) 96 (1972), 277-295.
\end{document}